\documentclass[a4paper,12pt]{article}
\usepackage{mathrsfs}
\usepackage{amsfonts}
\usepackage{amscd,color}
\usepackage{amsmath,amsfonts,amssymb,amscd}
\usepackage{indentfirst,graphicx,epsfig}
\usepackage{graphicx,psfrag}
\input{epsf}

\newtheorem{df}{Definition}[section]
\newtheorem{thm}{Theorem}[section]

\newtheorem{conjecture}{Conjecture}[section]
\newtheorem{lem}{Lemma}[section]

\newtheorem{claim}{Claim}[section]

\newenvironment {proof} {\noindent{\em Proof.}}{\hspace*{\fill}$\Box$}

\title{\bf On two conjectures about the\\ proper connection number of graphs\footnote{Supported by NSFC No.11371205 and 11531011, the ``973" program No.2013CB834204, and PCSIRT.}}

\author{\small Fei Huang, Xueliang Li, Zhongmei Qin\\
{\small  Center for Combinatorics and LPMC}\\
{\small Nankai University, Tianjin 300071, P.R. China}\\
{\small Email: huangfei06@126.com, lxl@nankai.edu.cn, qinzhongmei90@163.com}\\
{\small Colton Magnant}\\
{\small Department of Mathematical Sciences}\\
{\small Georgia Southern University, Statesboro, GA 30460-8093, USA}\\
{\small Email: cmagnant@georgiasouthern.edu}\\
{\small Kenta Ozeki}\\
{\small National institute of Informatics, Tokyo 101-8430, Japan}\\
{\small Email: ozeki@nii.ac.jp}
}
\date{}
\begin{document}
\maketitle
\begin{abstract}
A path in an edge-colored graph is called proper if no two consecutive edges of the path receive the same color. For a connected graph $G$, the proper connection number $pc(G)$ of $G$ is defined as the minimum number of colors needed to color its edges so that every pair of distinct vertices of $G$ are connected by at least one proper path in $G$. In this paper, we consider two conjectures on the proper connection number of graphs.
The first conjecture states that if $G$ is a noncomplete graph with connectivity $\kappa(G) = 2$ and minimum degree $\delta(G)\ge 3$, then $pc(G) = 2$, posed by Borozan et al.~in [Discrete Math. 312(2012), 2550-2560]. We give a family of counterexamples to disprove this conjecture. However, from a result of Thomassen it follows that 3-edge-connected noncomplete graphs have proper connection number 2. Using this result, we can prove that if $G$ is a 2-connected noncomplete graph with $diam(G)=3$, then $pc(G) = 2$, which solves the second conjecture we want to mention, posed by Li and Magnant in [Theory \& Appl. Graphs 0(1)(2015), Art.2].
 \\[2mm]
\end{abstract}

\section{Introduction}
All graphs in this paper are simple, finite and undirected. We follow \cite{BM} for graph theoretical notation and terminology not defined here.
Let $G$ be a connected graph with vertex set $V(G)$ and edge set $E(G)$. For $v \in V(G)$, let $N(v)$ denote the set of neighbors of $v$. For a subset $U \subseteq V(G)$, let $N(U)=\left(\bigcup_{v \in U}N(v)\right)\setminus U$. For any two disjoint subsets $X$ and $Y$ of $V(G)$, we use $E(X,Y)$ to denote the set of edges of $G$ that have one end in $X$ and the other in $Y$. Denote by $|E(X,Y)|$ the number of edges in $E(X,Y)$. An $(X, Y)$-path is a path which starts at a vertex of $X$, ends at a vertex of $Y$, and whose internal vertices belong to neither $X$ nor $Y$.

Let $G$ be a nontrivial connected graph with an {\it edge-coloring} $c : E(G)\rightarrow \{1, 2, \ldots, t\}$, $t \in \mathbb{N}$, where adjacent edges may have the same color. If adjacent edges of $G$ are assigned different colors by $c$, then $c$ is called a {\it proper (edge-)coloring}. For a graph $G$, the minimum number of colors needed in a proper coloring of $G$ is referred to as the {\it edge-chromatic number} of $G$ and denoted by $\chi'(G)$. A path of an edge-colored graph $G$ is said to be a {\it rainbow path} if no two edges on the path have the same color. The graph $G$ is called {\it rainbow connected} if for any two vertices there is a rainbow path of $G$ connecting them. An edge-coloring of a connected graph is a {\it rainbow connecting coloring} if it makes the graph rainbow connected. For a connected graph $G$, the \emph{rainbow connection number} $rc(G)$ of $G$ is defined to be the smallest number of colors that
are needed in order to make $G$ rainbow connected. The concept of rainbow connection of graphs was
introduced by Chartrand et al.~\cite{CJMZ} in 2008. Readers who are interested in this topic can
see \cite{LSS, LS} for a survey.

Motivated by the rainbow coloring and proper coloring in graphs, Andrews et al.~\cite{ALLZ} and Borozan et al.~\cite{BFGMMMT} introduced the concept of proper-path coloring. Let $G$ be a nontrivial connected graph with an edge-coloring. A path in $G$ is called a \emph{proper path} if no two adjacent edges of the path are colored with the same color. An edge-coloring of a connected graph $G$ is a \emph{proper-path coloring} if every pair of distinct vertices of $G$ are connected by a proper path in $G$. If $k$ colors are used, then $c$ is referred to as a {\it proper-path $k$-coloring}. An edge-colored graph $G$ is called {\it proper connected} if any two vertices of $G$ are connected by a proper path. For a connected graph $G$, the {\it proper connection number} of $G$, denoted by $pc(G)$, is defined as the smallest number of colors that are needed in order to make $G$ proper connected.

The proper connection of graphs has the following application background. When building a communication network between wireless signal towers, one fundamental requirement is that the network be connected. If there cannot be a direct connection between two towers $A$ and $B$, say for example if there is a mountain in between, there must be a route through other towers to get from $A$ to $B$. As a wireless transmission passes through a signal tower, to avoid interference, it would help if the incoming signal and the outgoing signal do not share the same frequency. Suppose that we assign a vertex to each signal tower, an edge between two vertices if the corresponding signal towers are directly connected by a signal and assign a color to each edge based on the assigned frequency used for the communication. Then, the number of frequencies needed to assign frequencies to the connections between towers so that there is always a path avoiding interference between each pair of towers is precisely the proper connection number of the corresponding graph.

Let $G$ be a nontrivial connected graph of order $n$ (number of vertices) and size $m$ (number of edges). Then the proper connection number of $G$ has the following clear bounds:
$$1\le pc(G)\le \min\{ rc(G), \chi'(G)\}\le m.$$
Furthermore, $pc(G) = 1$ if and only if $G = K_n$, $pc(G) = m$ if and only if $G = K_{1,m}$ is a star of size $m$.

Given an edge-colored path $P =v_1v_2\ldots v_{s-1}v_s$ between any two vertices
$v_1$ and $v_s$, we denote by $start(P)$ the color of the first edge in the path,
i.e., $c(v_1v_2)$, and by $end(P)$ the color of the last edge in the path, i.e., $c(v_{s-1}v_s)$.
If $P$ is just the edge $v_1v_s$, then $start(P)=end(P)=c(v_1v_s)$.

\begin{df}[\cite{BFGMMMT}]
Let $c$ be an edge-coloring of $G$ that makes $G$ proper connected. We say that $G$ has the \emph{strong property under $c$} if for any pair of vertices $u, v\in V(G)$, there exist two proper paths $P_1$, $P_2$ connecting them (not necessarily disjoint) such that $start(P_1)\neq start(P_2)$ and $end(P_1)\neq end(P_2)$.
\end{df}

Next we list the following three lemmas, which will be used in this work.

\begin{lem}[\cite{ALLZ}]\label{lem3}
If $G$ is a nontrivial connected graph and $H$ is a connected spanning
subgraph of $G$, then $pc(G) \leq pc(H)$. In particular, $pc(G) \leq pc(T)$
for every spanning tree $T$ of $G$.
\end{lem}

\begin{lem}[\cite{BFGMMMT}]\label{lem4}
If $G$ is a 2-connected graph with $n$ vertices, then $pc(G) \leq 3$.
Furthermore, there exists a 3-edge-coloring $c$ of $G$ such that $G$ has the strong property
under $c$.
\end{lem}

\begin{lem}[\cite{BFGMMMT, HLW}]\label{lem1}
If $G$ is a connected bridgeless bipartite graph with $n$ vertices, then $pc(G) \leq
2$. Furthermore, there exists a 2-edge-coloring $c$ of $G$ such that $G$ has the strong
property under $c$.
\end{lem}

\begin{lem}[\cite{ALLZ}]\label{lem2}
Let $G$ be a connected graph and $v$ a vertex not in $G$. If $pc(G)=2$, then $pc(G \cup v)=2$ as long as $d(v)\geq 2$, that is, we connect $v$ to $G$ by using at least two edges.
\end{lem}

For more details we refer to \cite{ALLZ,BFGMMMT,GLQ,HLQC,HLW1,LWY} and a dynamic survey \cite{LC}.

The first conjecture we will consider in this paper is as follows, which was posed by Borozan et al.~in \cite{BFGMMMT}.

\begin{conjecture}[\cite{BFGMMMT}]\label{conj1}
If $G$ is not a complete graph such that the connectivity $\kappa(G) = 2$ and the minimum degree $\delta(G)\ge 3$, then $pc(G) = 2$.
\end{conjecture}

The second conjecture we will consider is as follows, which was posed by Li and Magnant in \cite{LC}
\begin{conjecture}[\cite{LC}]\label{conj2}
If $G$ is a 2-connected noncomplete graph with $ diam(G)=3$, then $pc(G) = 2$.
\end{conjecture}

\begin{figure}
  \centering
   \scalebox{0.7}{\includegraphics{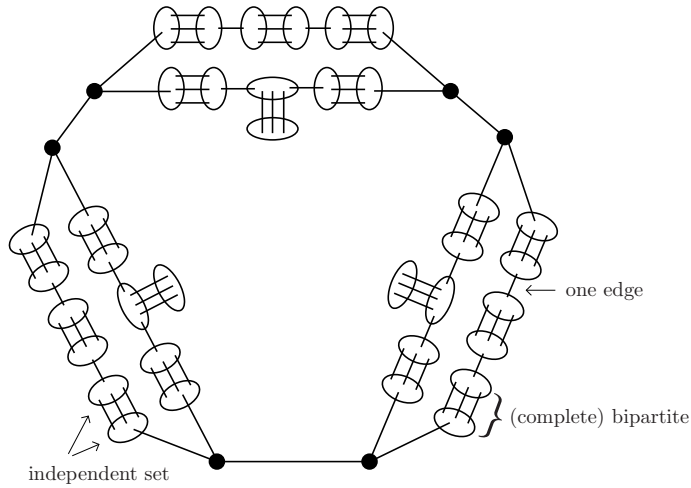}}\\
  \caption{Counterexamples for Conjecture~\ref{conj1} \label{Fig1}}\
\end{figure}

Unfortunately, Conjecture \ref{conj1} is not true. A family of counterexamples is shown in Figure~\ref{Fig1}. It is obvious that the graph in Figure 1 has connectivity 2 and minimum degree 3, however, we will show in next section that it has proper connection number 3, but not 2. From a result of Thomassen \cite{T} it follows immediately that 3-edge-connected noncomplete graphs have proper connection number 2. Using this result we can prove in Section~3 that Conjecture~\ref{conj2} holds true.

\begin{figure}[h,t,b,p]
  \centering
 \scalebox{0.8}{\includegraphics{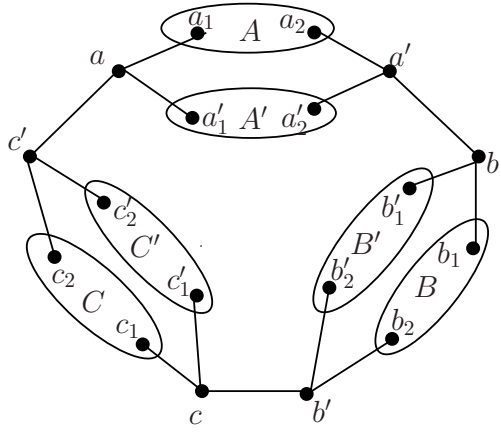}}\\
  \caption{Labeling of the graph $G$ \label{Fig2}}
\end{figure}

\section{Disprove Conjecture \ref{conj1}}

\begin{thm}
Let $G$ be a graph as shown in Figure~\ref{Fig1}. Then, $pc(G)=3$, which disproves Conjecture \ref{conj1}.
\end{thm}
\begin{proof}
First, we label the graph $G$ as in Figure~\ref{Fig2} for simplicity where subgraphs correspond to the subgraphs in Figure~\ref{Fig1}. From Lemma \ref{lem4}, we have that $pc(G) \le 3$. Thus, it is sufficient to show that $pc(G) \ne 2$. Assume, to the contrary, that we have a 2-edge-coloring $c$ which makes $G$ proper connected. Then, for any two vertices of $G$, there is a proper path connecting them.

\begin{claim}\label{Claim:OneWay}
In each pair $(A, A')$, $(B, B')$ and $(C, C')$, say $(A, A')$, there exists a vertex $v$ such that either all the proper paths from $v$ to $G \setminus (A \cup A')$ go through the edge $ac'$ rather than $a'b$ or all the proper paths from $v$ to $G \setminus (A \cup A')$ go through the edge $a'b$ rather than $ac'$.
\end{claim}

\begin{proof}
Suppose, to the contrary, that every vertex in $A \cup A'$ has proper paths to $G \setminus (A \cup A')$ through $ac'$ and also has proper paths to $G \setminus (A \cup A')$ through $a'b$. Let $f=v_1v_2$ and $f'=v_3v_4$ be the two cut-edges in $G[A]$ with $f$ the closer edge to $a$ and $f'$ the closer edge to $a'$. Let $e=u_1u_2$ and $e'=u_3u_4$ be the two cut-edges in $G[A']$ with $e$ the closer edge to $a$ and $e'$ the closer edge to $a'$. Also assume the vertices with lower index on each of these edges are closer to $a$.

If $aa_1$ and $f$ have different colors, then all the proper paths from $v_2$ to $G \setminus A $ must go through $a_2a'$. Thus, $a'a_2'$ and $a'b$ have the same color. So $a_2'$ has no proper path to $G \setminus (A \cup A')$ through $a'b$, a contradiction. Hence, $aa_1$ and $f$ have the same color. Similarly, $f$ and $f'$ as well as $f'$ and $a_2a'$ have the same color. Thus, $aa_1$, $f$, $f'$ and $a_2a'$ all have the same color. If $a'a_2$ and $a'b$ have the same color, then there is no proper path from $a_{2}$ to $G \setminus (A \cup A')$ through $a'b$ since the parity of any path from $a$ to $a'$ passing through $A$ is different from the parity of any path from $a$ to $a'$ passing through $A'$. Thus, $a'a_2$ and $a'b$ have different colors and symmetrically, $aa_1$ and $ac'$ have different colors.

If $e$ has the opposite color from $aa_{1}'$, then $u_{2}$ cannot possibly have a proper path leaving $A'$ through the edge $a_{1}'a$. Then, since $a_{2}a'$ and $a'b$ have different colors, the edge $a_{2}'a'$ must share one of those colors. This means that $u_{2}$ has a proper path to $G \setminus (A \cup A')$ through only one of $ac'$ or $a'b$, a contradiction. On the other hand, this means that $e$ must have the same color as $aa_{1}'$. Similarly, we have that $e'$ must have the same color as $a'a_{2}'$. If $e$ and $e'$ have the same color, then $u_{1}$ cannot possibly have a proper path leaving $A'$ through the edge $a_{2}'a$. Thus, $u_{1}$ has no proper path to $G \setminus (A \cup A')$ through one of $ac'$ or $a'b$, a contradiction. Now we have that $e$ and $e'$ have different colors. By the symmetry, we suppose that $aa_1'$ and $aa_1$ have the same color and $a'a_2'$ and $a'b$ have the same color. Thus, any vertex of $A'$ has no proper path to $G \setminus (A \cup A')$ through $a'b$, which is a contradiction. So, there exists a vertex $v \in A \cup A'$ such that either all the proper paths from $v$ to $G \setminus (A \cup A')$ go through the edge $ac'$ rather than $a'b$ or all the proper paths from $v$ to $G \setminus (A \cup A')$ go through the edge $a'b$ rather than $ac'$. Of course, the same argument holds in $B \cup B'$ and $C \cup C'$ to complete the proof.
\end{proof}

Let $a^{*}$ be the vertex in $A \cup A'$ resulting from Claim~\ref{Claim:OneWay} and similarly define $b^{*}$ and $c^{*}$. By the pigeonhole principle, there exists a pair of these vertices that leave their respective sets in the same direction. More specifically, we may assume without loss of generality that all the proper paths from $a^{*}$ to $G \setminus (A \cup A')$ through only $ac'$ (and not $a'b$) and all the proper paths from $b^{*}$ to $G \setminus (B \cup B')$ through only $ba'$ (and not $b'c$). Then there can be no proper path from $a^{*}$ to $b^{*}$.
\end{proof}\\

\section{Proof of Conjecture~\ref{conj2}}

Thomassen \cite{T} observed that given a graph $G$ which is at least $(2k-1)$-edge-connected, then $G$ contains a bipartite spanning subgraph $H$ for which $H$ is $k$-edge-connected. Combining with Lemma \ref{lem1}, we have the following Theorem.

\begin{thm}\label{th0}
If $G$ is a 3-edge-connected noncomplete graph, then $pc(G) = 2$ and there exists a 2-edge-coloring $c$ of $G$ such that $G$ has the strong property under $c$.
\end{thm}

In the following, we will use Theorem~\ref{th0} to give a confirmative proof for Conjecture~\ref{conj2}.

\begin{thm}
If $G$ is a 2-connected noncomplete graph with $ diam(G)=3$, then $pc(G)=2$.
\end{thm}

\begin{proof}
If $G$ is 3-edge-connected, Theorem~\ref{th0} implies that $pc(G)=2$. So, we may assume $\kappa'(G) = 2$, where $\kappa'(G)$ denotes the edge-connectivity of $G$. We distinguish the following two cases to proceed the proof.

\begin{figure}[h,t,b,p]
  \centering
 \scalebox{1}{\includegraphics{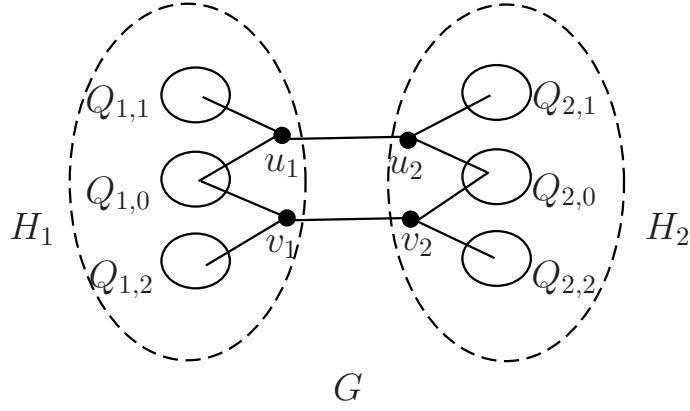}}\\
  \caption{The graph $G$ in Case 1}
\end{figure}

\textbf{Case 1:} There is a 2-edge-cut $S$ of $G$ such that each component of $G-S$ has at least three vertices.

Then, let $S = \{u_1u_2, v_1v_2\}$ be the 2-edge-cut of $G$ and $H_1, H_2$ be the components of
$G \backslash S$ such that $|H_i| \ge 3$ for $i=1,2$, where $u_1,v_1 \in H_1$ and $u_2,v_2 \in H_2$. Since $G$ is 2-connected, we have that $u_1 \ne v_1$ and $u_2 \ne v_2$. Let $Q_i=H_i\setminus \{u_i,v_i\}$ for $i=1,2$. Since $diam(G) = 3$, we know that any vertex $w\in Q_i$ must be adjacent to at least one vertex of $\{u_i,v_i\}$ for $i=1,2$. For each $Q_i$ ($i=1,2$), define the subsets $Q_{i,0} = N(u_i) \cap N(v_i) \cap Q_i$, $Q_{i,1} = (N(u_i) \cap Q_i)\backslash Q_{i,0}$ and $Q_{i,2} = (N(v_i) \cap Q_i)\backslash Q_{i,0}$. Since $diam(G)=3$, we have that at least one of $Q_{1,1}$ and $Q_{2,2}$ is empty. Without loss of generality, we may assume that $Q_{2,2}$ is empty. Similarly, at least one of $Q_{1,2}$ and $Q_{2,1}$ is empty. So, there are two subcases to deal with.

\textbf{Subcase 1.1:} $Q_{1,2}$ is empty.

If $Q_{1,0}$ is empty, then $u_1$ is a cut-vertex of $G$, a contradiction. So, $Q_{1,0}$ is nonempty. Similarly, $Q_{2,0}$ is nonempty.
Let $G_0=G[\{u_1,v_1,u_2,v_2\}\cup Q_{1,0}\cup Q_{2,0}]$. The graph $G_0$ contains a 2-connected bipartite spanning subgraph. So, $pc(G_0)=2$ from Lemmas \ref{lem3} and \ref{lem1}.

Let $G_1$ be a subgraph of $G$ obtained by adding a vertex to $G_0$ which has at least 2 edges into $G_0$. Furthermore, let $G_i$ be a
subgraph of $G$ obtained by adding a vertex to $G_{i-1}$ which has at least 2 edges connecting to $G_{i-1}$. By Lemma \ref{lem2}, $pc(G_i) = 2$ for all
$i$. We claim that such a sequence of subgraphs of $G$ exists, and we can find a spanning subgraph of $G$ by repeating this procedure. In order to
prove this, suppose that $G_i$ is the largest such subgraph of $G$ and suppose that there exists a vertex $v \in  G \setminus G_i$. Assume, without loss of generality, that $v\in Q_{1,1}$. Since $G$ is 2-connected, we have that there is a 2-fan from $v$ to $G_i$. So we can find a path from $v$ to $G_i$ other than $vu_1$ within $H_1$. Let $w$ be the last vertex on this path which is not in $G_i$. We know that $w$ must be adjacent to $u_1$. This means that $d_{G_i}(w) \ge 2$, and so we may set $G_{i+1} = G_i \cup w$ to get a contradiction. This completes the proof.

\textbf{Subcase 1.2:} $Q_{2,1}$ is empty.

If $E(Q_{1,1}, Q_{1,2})= \emptyset$, then we can get $pc(G)=2$ by a similar argument as in Case 1. Now we have $|E(Q_{1,1}, Q_{1,2})|\ge 1$. Let $M$ be a maximum matching of $E(Q_{1,1}, Q_{1,2})$, and let $G_0'=G[\{u_1,u_2,v_1,v_2\} \cup M \cup Q_{2,0}]$. Note that $G[\{u_2,v_2\} \cup Q_{2,0}]$ contains a 2-connected bipartite spanning subgraph if $|Q_{2,0}|\ge 2$. From Lemma \ref{lem1} we know that $G[\{u_2,v_2\} \cup Q_{2,0}]$ has a 2-edge-coloring with the strong property. If $|M| \ge 2$, then $G[M \cup \{u_1, v_1\}]$ contains a 2-connected bipartite spanning subgraph, and we can also get that $G[M \cup \{u_1, v_1\}]$ has a 2-edge-coloring with the strong property by Lemma \ref{lem1}. It is easy to check that $pc(G_0')=2$.

Let $G_1'$ be a subgraph of $G$ obtained by adding a vertex to $G_0'$ which has at least 2 edges connecting to $G_0'$. Furthermore, let $G_i'$ be a subgraph of $G$ obtained by adding a vertex to $G_{i-1}'$ which has at least 2 edges connecting to $G_{i-1}'$. By Lemma \ref{lem2}, $pc(G_i') = 2$ for all
$i$. We claim that such a sequence of subgraphs of $G$ exists, and we can find a spanning subgraph of $G$ by repeating this procedure. In order to
prove this, suppose that $G_i'$ is the largest such subgraph of $G$ and suppose that there exists a vertex $v \in  G \setminus G_i'$. Let $Q_{1,1}'=Q_{1,1}\setminus V(G_i')$, $Q_{1,2}'=Q_{1,2}\setminus V(G_i')$. According to the construction of $G_0'$, we have $E(Q_{1,1}', Q_{1,2}')=\emptyset$. Certainly, every vertex adjacent to both $u_1$ and $v_1$ is in $G_i'$. This means $v \in Q_{1,1}' \cup Q_{1,2}'$. Without loss of generality, we assume $v \in Q_{1,1}'$. Since $G$ is 2-connected, we have that there is a 2-fan from $v$ to $G_i'$. So, we can find a path from $v$ to $G_i'$ other than $vu_1$ within $H_1$. Let $w$ be the last vertex on this path which is not in $G_i'$. We know that $w$ must be adjacent to $u_1$. This means that $d_{G_i'}(w) \ge 2$, and so we may set $G_{i+1}' = G_i' \cup w$ to get a contradiction. This completes the proof.

\textbf{Case 2:} For every 2-edge-cut $S$ of $G$, $G-S$ has a component with at most two vertices.

If $G$ does not have even cycle, then $G=C_3$, $C_5$ or $C_7$ since $G$ is 2-connected and $diam(G)=3$. It follows that $pc(G)=2$. Thus, we suppose $G$ contains an even cycle. Let $H=H(U,V)$ be a maximal 2-edge-connected bipartite subgraph of $G$. We claim that $H$ contains all the vertices with degree at least 3 of $G$. Assume, to the contrary, that there is a vertex $v \in V(G)\setminus V(H)$ with $d_G(v) \ge 3$. We know that there exist three edge-disjoint $(v,H)$-paths. If this is not the case, then we can find a 2-edge-cut $S'$ of $G$ such that each component of $G-S'$ has at least three vertices, which contradicts the assumption. 
By the pigeonhole principle, there is a pair of $(v, H)$-paths with the same parity of number of edges. Using these two paths, we can get a 2-edge-connected bipartite subgraph $H'$ containing the vertex $v$ and $H$, which contradicts to the maximality of $H$.

If $H$ is a spanning subgraph of $G$, then $pc(G)=2$ by Lemmas~\ref{lem3} and~\ref{lem1}. Otherwise, the components of $G-H$ has the following two types: (1) an isolated vertex; (2) an edge. Let $A_1, \ldots, A_p, B_1, \ldots, B_q$ be the components of $G-H$ such that $|A_i|=1 \ (1 \le i \le p)$ and $|B_j|=2 \ (1 \le j \le q)$, where $p, q$ are nonnegative integers, and $p=0$ or $q=0$ means that there is no $A_i$-type component or $B_j$-type component. Let $N(B_j)=\{a_j, b_j\}$. Then $a_j \ne b_j$ since $G$ is 2-connected. If $a_j$ and $b_j$ are in different partite sets of $H$, then $B_j \cup H$ is also a 2-edge-connected bipartite graph, which contradicts to the maximality of $H$. So, for each $B_j$ we have that $a_j$ and $b_j$ are in the same partite set of $H$. Let $C(a, b)=\{B_i | N(B_i)=\{a,b\}, 1\le i \le q\}$. Since $H$ is a 2-edge-connected bipartite graph, it follows from Lemma \ref{lem1} that $H$ has a 2-edge-coloring $c$ which makes $H$ have the strong property under $c$. If $|C(a,b)| \ge 2$, then $G[V(C(a,b)) \cup \{a,b\}]-ab$ is a 2-edge-connected bipartite graph. Thus, there is a 2-edge-coloring $c$ such that $G[V(C(a,b)) \cup \{a,b\}]-ab$ has the strong property under $c$ by Lemma \ref{lem1}. Now we color the edges of $G \setminus \{A_1,  \ldots, A_p\}$ with two colors $\{1, 2\}$. Firstly, we color the edges of $H$ such that $H$ has the strong property under this coloring. Then, we color the edges of $G[U]$ and $G[V]$ with color 2. If $|C(a,b)| \ge 2$, then we color the edges of $G[V(C(a,b)) \cup \{a,b\}]-ab$ such that $G[V(C(a,b)) \cup \{a,b\}]-ab$ has the strong property under this coloring. If $|C(a,b)|=1$, then $G[V(C(a,b)) \cup \{a,b\}]-ab$ is a path $P$ with length 3. Thus, we color the two pendant edges of $P$ with color 1 and the central edge of $P$ with color 2.

Next, we will show that this 2-edge-coloring $c$ makes $G \setminus \{A_1,  \ldots, A_p\}$ proper connected. Let $u,v$ be any two vertices of $G \setminus \{A_1,  \ldots, A_p\}$. If both $u $ and $v$ are in $H$, then there is already a proper path connecting them in $H$. If one of $u, v$ is in $H$, without loss of generality, let $u \in H$, then $v$ has a neighbor $v'$ in $H$. Since $H$ has the strong property under $c$, it follows that there is a proper path $P$ connecting $u$ and $v'$ in $H$ such that $end(P)\ne c(vv')$, and $P\cup \{vv'\}$ is a proper path connecting $u, v$. If $\{u, v\} \in V(C(a,b))$, then there is already a proper path connecting them in $G[V(C(a,b)) \cup \{a,b\}]-ab$. Suppose that $u \in V(C(a,b))$ and $v \in V(C(a', b'))$. Since $diam(G)=3$, we have $E(\{a,b\}, \{a',b'\}) \ne \emptyset$. If $\{a, b , a' ,b'\} \subseteq U$ or $V$, then without loss of generality, let $aa' \in E(\{a,b\}, \{a',b'\})$. If $b=b'$ and $|C(a,b)| \ge 2$, $|C(a', b')| \ge 2$, then it is easy to check that there is a proper path $P_1$ connecting $u$ and $b$ with $end(P_1)=1$ in $G[V(C(a,b)) \cup \{a,b\}]-ab$, and there is also a proper path $P_2$ connecting $v$ and $b'$ with $end(P_2)=2$ in $G[V(C(a',b')) \cup \{a',b'\}]-a'b'$. Thus, $P_1 \cup P_2$ is a proper path connecting $u$ and $v$. If $b \ne b'$ or $b=b'$ and $|C(a, b)|=1$ or $b=b'$ and $|C(a',b')|=1$, then it is easy to see that there is a proper path $P_1$ connecting $u$ and $a$ with $end(P_1)=1$ in $G[V(C(a,b)) \cup \{a,b\}]-ab$, and there is also a proper path $P_2$ connecting $v$ and $a'$ with $end(P_2)=1$ in $G[V(C(a',b')) \cup \{a',b'\}]-a'b'$. Thus, $P_1 \cup \{aa'\} \cup P_2$ is a proper path connecting $u$ and $v$. If $\{a, b\} \subseteq U$ and $\{a', b'\}\subseteq V$, then there is already a proper $(\{a,b\}, \{a', b'\})$-path $P$ in $H$ with $start(P)=end(P)=2$. Without loss of generality, let $a ,a'$ be the two endvertices of $P$. It is easy to check that there is a proper path $P_1$ connecting $u$ and $a$ with $end(P_1)=1$ in $G[V(C(a,b)) \cup \{a,b\}]-ab$, and there is also a proper path $P_2$ connecting $v$ and $a'$ with $end(P_2)=1$ in $G[V(C(a',b')) \cup \{a',b'\}]-a'b'$. Thus, $P_1 \cup P \cup P_2$ is a proper path connecting $u$ and $v$. Hence, $G \setminus \{A_1,  \ldots, A_p\}$ is proper connected. Since $d_G(v) = 2$ for each vertex $v$ of $A_i(1 \le i \le p)$, it follows that $pc(G)=2$ by Lemma~\ref{lem2}, completing the proof.
\end{proof}

\end{document}